\documentclass[a4paper,12pt]{article}
\usepackage{latexsym,amssymb}
\usepackage{amsmath}
\usepackage[mathscr]{eucal}
\usepackage{epic,eepic}
\usepackage{color}
\usepackage{enumerate}

\if0%%%
\usepackage[anchorcolor=blue,%
 bookmarks=true,%
 bookmarksnumbered=true,%
 colorlinks=true,%
 citecolor=cyan,%
 linkcolor=blue,%
 dvipdfmx%
]{hyperref}
\fi%%%

\usepackage{amsthm}
\newtheorem{theorem}{Theorem}[section]
\newtheorem{proposition}[theorem]{Proposition}
\newtheorem{lemma}[theorem]{Lemma}

\theoremstyle{definition}

\theoremstyle{remark}
\newtheorem{remark}[theorem]{Remark}
\newtheorem{example}[theorem]{Example}

\makeatletter

\@addtoreset{equation}{section}
\makeatother

\newcommand{\N}{\mathbb{N}}
\newcommand{\R}{\mathbb{R}}
\newcommand{\Hyp}{\mathbb{H}}
\newcommand{\Sph}{\mathbb{S}}

\newcommand{\sJ}{\mathsf{J}}

\newcommand{\ve}{\varepsilon}
\newcommand{\del}{\partial}
\newcommand{\lra}{\longrightarrow}

\newcommand{\dd}{\mathrm{d}}

\newcommand{\ol}[1]{\overline{#1}}

\newcommand{\CAT}{\mathrm{CAT}}

\newcommand{\RCD}{\mathrm{RCD}}

\newcommand{\diam}{\mathrm{diam}}

\def\argmin{\mathop{\mathrm{arg\, min}}}

\begin{document}

\title{Discrete-time gradient flows for unbounded convex functions on Gromov hyperbolic spaces}

\author{Shin-ichi OHTA\thanks{
Department of Mathematics, Osaka University, Osaka 560-0043, Japan
({\sf s.ohta@math.sci.osaka-u.ac.jp});
RIKEN Center for Advanced Intelligence Project (AIP),
1-4-1 Nihonbashi, Tokyo 103-0027, Japan}}

\date{\today}
%\date{\empty}
\maketitle

\begin{abstract}
In proper, geodesic Gromov hyperbolic spaces,
we investigate discrete-time gradient flows via the proximal point algorithm
for unbounded Lipschitz convex functions.
Assuming that the target convex function has negative asymptotic slope along some ray
(thus unbounded below),
we first prove the uniqueness of such a negative direction in the boundary at infinity.
Then, we show that the discrete-time gradient flow from an arbitrary initial point
diverges to that unique direction of negative asymptotic slope.
This is inspired by and generalizes results of Karlsson--Margulis and Hirai--Sakabe on nonpositively curved spaces
and a result of Karlsson concerning semi-contractions on Gromov hyperbolic spaces.
We also give an estimate of the rate of convergence
based on a contraction property for the proximal (resolvent) operator established in our previous work.
\end{abstract}

%%%%%%%%%%%%%%%%%%%%%%%%%%%%%%%%%%%%

\section{Introduction}%%%%%
%%%%%%%%%%%%%%%%%%%%%%%%%%%%%%%

This article is a continuation of \cite{Ohyp} concerning convex optimization on Gromov hyperbolic spaces.
We shall study the \emph{discrete-time gradient flow} for a convex function $f$
on a metric space $(X,d)$ built of the \emph{proximal} (or \emph{resolvent}) \emph{operator}
\begin{equation}\label{eq:Jf}
\sJ^f_{\tau}(x) :=\argmin_{y \in X} \bigg\{ f(y) +\frac{d^2(x,y)}{2\tau} \bigg\},
\end{equation}
where $\tau>0$ is the step size.
Iterating $\sJ^f_{\tau}$ is a well known scheme to construct a continuous-time gradient flow for $f$
in the limit as $\tau \to 0$.
Generalizations of the theory of gradient flows to convex functions on metric spaces
have been making impressive progress since 1990s,
including those on $\CAT(0)$-spaces \cite{AGSbook,Babook,Jo,Ma}
(see \cite{HH,Hi2,HS} for some applications to optimization theory),
$\CAT(1)$-spaces \cite{OP1,OP2},
Alexandrov spaces and the Wasserstein spaces over them \cite{Ly,Ogra,OP1,Pe,Sa},
and $\RCD(K,\infty)$-spaces \cite{St}.
These spaces are, however, all Riemannian in the sense that they exclude non-Riemannian Finsler manifolds
(in particular, non-inner product normed spaces).
In fact, despite great success for Riemannian spaces,
much less is known for non-Riemannian spaces (even for normed spaces);
we refer to \cite{Onc,OSnc} for the failure of the contraction property.
Motivated by this large gap and a fact that
some non-Riemannian Finsler manifolds can be Gromov hyperbolic (Example~\ref{ex:Ghyp}\eqref{it:Hil}),
we initiated investigation of convex optimization on Gromov hyperbolic spaces in \cite{Ohyp}
(see also \cite{Olln} for a related study on barycenters in Gromov hyperbolic spaces).

The \emph{Gromov hyperbolicity}, introduced in a seminal work \cite{Gr} of Gromov,
is a notion of negative curvature of large scale.
A metric space $(X,d)$ is said to be Gromov hyperbolic
if it is \emph{$\delta$-hyperbolic} for some $\delta \ge 0$ in the sense that
\begin{equation}\label{eq:Ghyp}
(x|z)_p \ge \min\{ (x|y)_p, (y|z)_p \} -\delta
\end{equation}
holds for all $p,x,y,z \in X$, where
\[ (x|y)_p:=\frac{1}{2} \big\{ d(p,x)+d(p,y)-d(x,y) \big\} \]
is the \emph{Gromov product}.
If \eqref{eq:Ghyp} holds with $\delta=0$,
then the quadruple $p,x,y,z$ is isometrically embedded into a tree.
Therefore, the $\delta$-hyperbolicity means that $(X,d)$ is close to a tree
up to local perturbations of size $\delta$ (cf.\ Example~\ref{ex:Ghyp}\eqref{it:tree}).
Admitting such local perturbations is a characteristic feature of the Gromov hyperbolicity.

In the investigation of gradient flows in Gromov hyperbolic spaces,
we employ discrete-time gradient flows of large time step $\tau$ (``giant steps''),
because of the inevitable local perturbations.
Precisely, for a convex function $f\colon X \lra \R$ and an arbitrary initial point $x_0 \in X$,
we study the behavior of recursive applications of the proximal operator \eqref{eq:Jf}:
\begin{equation}\label{eq:x_k}
x_k \in \sJ^f_{\tau}(x_{k-1}), \quad k \in \N.
\end{equation}
The resulting sequence $x_0, x_1, x_2, \ldots$ can be regarded as a discrete approximation
of a (continuous-time) gradient curve for $f$ starting from $x_0$.
If $f$ is bounded below, then it is natural to expect that $x_k$ converges to a minimizer of $f$;
we refer to, e.g., \cite{Ba,OP1} for such convergence results
in metric spaces with upper or lower curvature bounds.
In \cite[Theorem~1.1]{Ohyp}, we established that $x_k$ is closer to a minimizer than $x_{k-1}$,
up to an additional term depending on the hyperbolicity constant $\delta$.
The estimates in \cite[Theorems~1.1, 1.3]{Ohyp} can be thought of as contraction properties akin to tress,
and were the first contraction estimates
concerning gradient flows for convex functions on non-Riemannian spaces.

In this article, inspired by Hirai--Sakabe's recent work \cite{HS},
we consider the case where the target convex function is unbounded below.
Then, instead of convergence to a minimizer,
we study divergence to the steepest direction in the boundary at infinity.
Note that the structure of boundary at infinity $\del X$, constructed as equivalence classes of rays in $X$,
is well investigated under both nonpositive curvature and Gromov hyperbolicity.
This kind of divergent phenomenon has been established by Karlsson--Margulis \cite{KM}
for \emph{semi-contractions} $\phi \colon X \lra X$ (namely $d(\phi(x),\phi(y)) \le d(x,y)$)
on nonpositively curved metric spaces $(X,d)$.
Precisely, their \emph{multiplicative ergodic theorem} \cite[Theorem~2.1]{KM} asserts that,
for a complete, uniformly convex metric space $(X,d)$
of nonpositive curvature in the sense of Busemann and almost every $x_0 \in X$,
there exists a ray $\xi \colon [0,\infty) \lra X$ such that
\[ \lim_{k \to \infty} \frac{d(\xi(\alpha k),x_k)}{k} =0, \]
where $x_k:=\phi(x_{k-1})$, provided that $\alpha:=\lim_{k \to \infty}d(x_0,x_k)/k >0$.
We refer to \cite[Proposition~4.2]{CL} for an application to
continuous-time gradient flows for Lipschitz convex functions on $\CAT(0)$-spaces,
and to \cite[Proposition~5.1]{Ka} for a generalization to semi-contractions on Gromov hyperbolic spaces.

The multiplicative ergodic theorem in \cite{KM} is applicable to the proximal operator on $\CAT(0)$-spaces
(thanks to the semi-contractivity of $\sJ^f_\tau$; see \cite[Theorem~2.2.22]{Babook}),
while in \cite[\S 3.2]{HS} they also studied gradient descent in Hadamard manifolds.
The results in \cite{HS} extract and discretize convex optimization ingredients
of the \emph{moment-weight inequality} for reductive group actions
by Georgoulas--Robbin--Salamon \cite{GRS} in geometric invariant theory,
and have applications to operator scaling problems (see also more recent \cite{Hi3}).
In Gromov hyperbolic spaces, we employ the proximal operator with (large) step size $\tau$;
however, it is not a semi-contraction (see Remark~\ref{rm:nocont}).
Nonetheless, analyzing the behavior of divergence to the boundary at infinity
and using the contraction estimate \cite[Theorem~1.1]{Ohyp}, we establish the following.

\begin{theorem}[Divergence to the steepest direction]\label{th:main}
Let $(X,d)$ be a proper, geodesic $\delta$-hyperbolic space with $\del X \neq \emptyset$,
and $f\colon X \lra \R$ be an $L$-Lipschitz convex function for some $L>0$.
Assume that $\alpha :=-\inf_{v \in \del X} \del_\infty f(v)>0$.
Then, we have the following.
\begin{enumerate}[{\rm (i)}]
\item\label{it:v*}
There exists a unique element $v_* \in \del X$ satisfying $\del_\infty f(v_*)<0$.
\item\label{it:to-v*}
For any $x_0 \in X$,
the discrete-time gradient curve $(x_k)_{k \in \N}$ as in \eqref{eq:x_k} converges to $v_* \in \del X$.
\item\label{it:quan}
Moreover, letting $\xi\colon [0,\infty) \lra X$ be a ray representing $v_*$ with $\xi(0)=x_0$,
we have
\begin{equation}\label{eq:quan}
\bigl( \xi(t)|x_k \bigr)_{x_0} \ge  k\sqrt{\tau}
 \biggl\{ \biggl( \frac{\alpha^2}{4L} +\frac{\alpha^4}{16L^3} \biggr) \sqrt{\tau} -2\sqrt{2L\delta} \biggr\}
\end{equation}
for all $k \in \N$ and $t>0$ satisfying $f(\xi(t)) \le f(x_k)$.
\end{enumerate}
\end{theorem}

Here, $\del_\infty f(v)$ is the \emph{asymptotic slope} defined by
\[ \del_\infty f(v) :=\lim_{t \to \infty} \frac{f(\xi(t))}{t} \]
for a ray $\xi \colon [0,\infty) \lra X$ representing $v \in \del X$.
We remark that the uniqueness in \eqref{it:v*} is a specific feature of the negative curvature;
it is not the case for $\CAT(0)$-spaces (see Example~\ref{ex:Buse}).
Note also that \cite[\S 3.2]{HS} includes finer analysis for gradient descent
under a concavity condition on $f$ (called the \emph{$L$-smoothness}).

The right hand side of \eqref{eq:quan} is positive (and hence \eqref{eq:quan} is nontrivial) if
\begin{equation}\label{eq:tau}
\tau > \frac{2^{11} L^7}{(4L^2 +\alpha^2)^2 \alpha^4}\delta.
\end{equation}
Since $L \ge \alpha$, for example,
\[ \tau > \frac{82L^7}{\alpha^8}\delta \]
is sufficient for \eqref{eq:tau}.
Then, \eqref{eq:quan} yields a quantitative estimate for the divergence of $(\xi(t)|x_k)_{x_0}$ to infinity,
that can be regarded as an estimate of the rate of convergence of $x_k$ to $v_*$ (see Remark~\ref{rm:bdry}).

\begin{remark}[On Lipschitz continuity]\label{rm:Lip}
It follows from the convexity of $f$ that the (closed) sublevel set $X':=f^{-1}((-\infty,f(x_0)])$ is convex,
and hence $(X',d)$ is also a proper, geodesic $\delta$-hyperbolic space.
For this reason, in Theorem~\ref{th:main}, it is in fact sufficient to assume that $f$ is $L$-Lipschitz on $X'$.
\end{remark}

This article is organized as follows.
After preliminaries in Section~\ref{sc:pre} for Gromov hyperbolic spaces,
Section~\ref{sc:cvx} is devoted to analysis of the asymptotic slope
of unbounded convex functions, including the proof of Theorem~\ref{th:main}\eqref{it:v*}.
In Section~\ref{sc:gf}, we study discrete-time gradient flows
and prove Theorem~\ref{th:main}\eqref{it:to-v*}, \eqref{it:quan}.

\section{Preliminaries for Gromov hyperbolic spaces}\label{sc:pre}%%%%%
%%%%%%%%%%%%%%%%%%%%%%%%%%%%%%%

For $a,b \in \R$, we set $a \wedge b:=\min\{a,b\}$ and $a \vee b:=\max\{a,b\}$.
Besides Gromov's original paper \cite{Gr}, we refer to \cite{Bo,BH,DSU,Va}
for the basics and various applications of the Gromov hyperbolicity.

\subsection{Gromov hyperbolic spaces}\label{ssc:Ghyp}%%%%%
%%%%%%%%%%%%%%%%%%%%%%%%%%%%%%%

Let $(X,d)$ be a metric space.
Recall from the introduction that the \emph{$\delta$-hyperbolicity} for $\delta \ge 0$ is defined by
\begin{equation}\label{eq:d-hyp}
(x|z)_p \ge (x|y)_p \wedge (y|z)_p -\delta
\end{equation}
for all $p,x,y,z \in X$, where
\[ (x|y)_p :=\frac{1}{2} \big\{ d(p,x) +d(p,y) -d(x,y) \big\} \]
is the \emph{Gromov product}.
Since the triangle inequality implies
\[ 0 \le (x|y)_p \le d(p,x) \wedge d(p,y), \]
the Gromov product does not exceed the diameter $\diam(X):=\sup_{x,y \in X}d(x,y)$.
Hence, if $\diam(X) \le \delta$, then $(X,d)$ is $\delta$-hyperbolic.
This also means that the local structure of $X$ (up to size $\delta$)
is not influential in the $\delta$-hyperbolicity.

Another important fact is that trees are $0$-hyperbolic; in this sense,
a $\delta$-hyperbolic space is close to a tree up to an additive constant $\delta$.
The Gromov hyperbolicity can also be regarded as a large-scale notion of negative curvature.
Let us mention some typical examples.

\begin{example}\label{ex:Ghyp}
\begin{enumerate}[(a)]
\item\label{it:CAT}
Complete, simply connected Riemannian manifolds of sectional curvature $\le -1$
(or, more generally, \emph{$\CAT(-1)$-spaces})
are Gromov hyperbolic (see, e.g., \cite[Proposition~H.1.2]{BH}).

\item\label{it:Hil}
An important difference between $\CAT(-1)$-spaces and Gromov hyperbolic spaces
is that the latter admits some non-Riemannian Finsler manifolds.
For instance, \emph{Hilbert geometry} on a bounded convex domain in the Euclidean space
is Gromov hyperbolic under mild convexity and smoothness conditions
(see \cite{KN}, \cite[\S 6.5]{Obook}).

\item\label{it:Tei}
For the \emph{Teichm\"uller space} of a surface of genus $g$ with $p$ punctures,
the \emph{Weil--Petersson metric} (which is incomplete, Riemannian)
is known to be Gromov hyperbolic if and only if $3g-3+p \le 2$ \cite{BF},
whereas the \emph{Teichm\"uller metric} (which is complete, Finsler)
does not satisfy the Gromov hyperbolicity \cite{MW}
(see also \cite[\S 6.6]{Obook}).

\item\label{it:disc}
The definition \eqref{eq:d-hyp} makes sense for discrete spaces,
and the Gromov hyperbolicity has found rich applications in group theory.
A discrete group whose Cayley graph satisfies the Gromov hyperbolicity
is called a \emph{hyperbolic group}; we refer to \cite{Bo,Gr}, \cite[Part~III]{BH}.

\item\label{it:tree}
Suppose that $(X,d)$ admits an isometric embedding $\phi \colon T \lra X$ from a tree $(T,d_T)$
such that the $\delta$-neighborhood $B(\phi(T),\delta)$ of $\phi(T)$ covers $X$.
Then, since $(T,d_T)$ is $0$-hyperbolic, we can easily see that $(X,d)$ is $6\delta$-hyperbolic.
\end{enumerate}
\end{example}

We call $(X,d)$ a \emph{geodesic space}
if any two points $x,y \in X$ are connected by a (minimal) \emph{geodesic}
$\gamma \colon [0,1] \lra X$ satisfying $\gamma(0)=x$, $\gamma(1)=y$,
and $d(\gamma(s),\gamma(t))=|s-t| d(x,y)$ for all $s,t \in [0,1]$.
In this case, there are a number of characterizations of the Gromov hyperbolicity,
most notably by the \emph{$\delta$-slimness} of geodesic triangles (see, e.g., \cite[\S III.H.1]{BH}).
We remark that, by \cite[Theorem~4.1]{BS},
every $\delta$-hyperbolic space can be isometrically embedded
into a complete, geodesic $\delta$-hyperbolic space.

In a $\delta$-hyperbolic space,
for any $x,y,z \in X$ and any geodesic $\gamma \colon [0,1] \lra X$ from $y$ to $z$, we have
\begin{equation}\label{eq:dxg}
d(x,\gamma) -2\delta \le (y|z)_x \le d(x,\gamma),
\end{equation}
where $d(x,\gamma) :=\min_{s \in [0,1]} d(x,\gamma(s))$ (see \cite[2.33]{Va}).
As a corollary, one can readily see the following; we give a proof for completeness
(cf.\ \cite[Lemma~III.H.1.15]{BH}).

\begin{lemma}\label{lm:2geod}
Let $(X,d)$ be a $\delta$-hyperbolic space and $x,y \in X$.
Then, for any two geodesics $\gamma,\eta \colon [0,1] \lra X$ from $x$ to $y$, we have
\[ d\bigl( \gamma(s),\eta(s) \bigr) \le 4\delta \]
for all $s \in [0,1]$.
\end{lemma}

\begin{proof}
Fix $s \in [0,1]$ and take $s' \in [0,1]$ such that $d(\gamma(s),\eta) =d(\gamma(s),\eta(s'))$.
It follows from \eqref{eq:dxg} and $(x|y)_{\gamma(s)}=0$ that $d(\gamma(s),\eta(s')) \le 2\delta$.
For $s'' \in [0,1]$ with $|s-s''|d(x,y) >2\delta$, the triangle inequality implies
\[ d\bigl( \gamma(s),\eta(s'') \bigr) \ge \bigl| d\bigl( x,\gamma(s) \bigr) -d\bigl( x,\eta(s'') \bigr) \bigr|
 =|s-s''|d(x,y) >2\delta. \]
Hence, we find $|s-s'|d(x,y) \le 2\delta$ and
\[ d\bigl( \gamma(s),\eta(s) \bigr) \le d\bigl( \gamma(s),\eta(s') \bigr) +|s-s'|d(x,y)
 \le 2\delta +2\delta =4\delta. \]
\end{proof}

\subsection{Gromov boundary}\label{ssc:bdry}%%%%%
%%%%%%%%%%%%%%%%%%%%%%%%%%%%%%%

Next, we introduce the \emph{Gromov boundary} $\del X$
of a proper, geodesic $\delta$-hyperbolic space $(X,d)$
(the \emph{properness} means that every bounded closed set is compact).
We refer to \cite[\S III.H.3]{BH} for further details,
as well as to \cite[\S 5]{Va} and \cite[\S 3.4]{DSU} for the more general non-proper, non-geodesic situation
(see Remark~\ref{rm:bdry} below).

A \emph{ray} $\xi\colon [0,\infty) \lra X$ is a geodesic of unit speed,
i.e., $d(\xi(s),\xi(t))=|s-t|$ for all $s,t \ge 0$.
Two rays $\xi,\zeta \colon [0,\infty) \lra X$ are said to be \emph{asymptotic} if
\[ \sup_{t \ge 0} d\bigl( \xi(t),\zeta(t) \bigr) <\infty. \]
Being asymptotic is an equivalence relation on the set of rays,
and we denote by $\del X$ the associated equivalence classes.
The equivalence class of a ray $\xi$ will be denoted by $\xi(\infty) \in \del X$.
For any $p \in X$ and $v \in \del X$, there exists a ray $\xi$ with $\xi(0)=p$ and $\xi(\infty)=v$
(see \cite[Lemma~III.H.3.1]{BH}).

We set $\ol{X}:=X \sqcup \del X$ (called the \emph{Gromov closure} or \emph{bordification} of $X$).
To endow $\ol{X}$ with a topology, we fix a point $p \in X$ and
consider geodesics $\xi \colon [0,l) \lra X$ of unit speed with $\xi(0)=p$.
If $l=\infty$, then $\xi$ is a ray.
If $l<\infty$, then we extend $\xi$ by letting $\xi(t):=\xi(l)$ for $t>l$, and put $\xi(\infty):=\xi(l)$.
In either case, we call $\xi\colon [0,\infty) \lra X$ a \emph{generalized ray}.
We say that two generalized rays $\xi,\zeta \colon [0,\infty) \lra X$ (emanating from $p$)
are equivalent if $\xi(\infty)=\zeta(\infty)$ (in $X$ or $\del X$).
Then, the set of equivalence classes of generalized rays is identified with $\ol{X}$.
A sequence $(\xi_i)_{i \in \N}$ of generalized rays is said to converge to
a generalized ray $\xi$ if $\xi_i$ converges to $\xi$ uniformly on each bounded interval in $[0,\infty)$.
This defines a topology of $\ol{X}$.
Note that, thanks to the properness, this topology restricted to $X$
coincides with the original topology of $X$ by the Arzel\`a--Ascoli theorem.

One can see that $\ol{X}$ is compact
(see \cite[Proposition~III.H.3.7]{BH}, \cite[Proposition~3.4.18]{DSU}).

\begin{proposition}\label{pr:cpt}
Let $(X,d)$ be a proper, geodesic Gromov hyperbolic space.
Then, $\ol{X}$ and $\del X$ are compact.
\end{proposition}

Moreover, $\ol{X}$ is metrizable, though we will not use it
(see \cite[Exercise~III.H.3.18(4)]{BH}, \cite[Proposition~3.6.13]{DSU}).

\begin{remark}[Gromov sequences]\label{rm:bdry}
Alternatively, one can introduce $\del X$ by considering sequences $(x_i)_{i \in \N}$ in $X$
such that $\lim_{i,j \to \infty}(x_i|x_j)_p =\infty$, where $p \in X$ is an arbitrarily fixed point.
Such a sequence is called a \emph{Gromov sequence}.
Two Gromov sequences $(x_i)_{i \in \N}$ and $(y_i)_{i \in \N}$
are defined to be \emph{equivalent} if $\lim_{i \to \infty} (x_i|y_i)_p =\infty$.
Then, there exists a natural bijection from the equivalence classes of Gromov sequences to $\del X$
(see \cite[Lemma~III.H.3.13]{BH}).
Precisely, each Gromov sequence $(x_i)_{i \in \N}$ converges to a point in $\del X$,
and its inverse map is given by associating a ray $\xi$ with the Gromov sequence $(\xi(i))_{i \in \N}$.
We stress that, for non-proper, non-geodesic Gromov hyperbolic spaces,
these two notions of the boundary may not coincide
(see \cite[Remark~5.5]{Va}, \cite[Remark~3.4.4]{DSU}).
\end{remark}

\section{Unbounded convex functions}\label{sc:cvx}%%%%%
%%%%%%%%%%%%%%%%%%%%%%%%%%%%%%%

Let $(X,d)$ be a proper, geodesic $\delta$-hyperbolic space,
and $f\colon X \lra \R$ be an $L$-Lipschitz function for some $L>0$
(i.e., $|f(x)-f(y)| \le Ld(x,y)$ for all $x,y \in X$).
We say that $f$ is (\emph{weakly, geodesically}) \emph{convex} if,
for any pair of points $x,y \in X$, there is a geodesic $\gamma\colon [0,1] \lra X$ from $x$ to $y$ such that
\begin{equation}\label{eq:conv}
f\big( \gamma(s) \big) \le (1-s)f(x) +sf(y)
\end{equation}
for all $s \in [0,1]$.
We remark that, iteratively choosing a midpoint satisfying \eqref{eq:conv},
one can actually find a geodesic $\gamma$ such that $f \circ \gamma$ is convex on $[0,1]$.

We are interested in the case where $\inf_X f=-\infty$.
Then, to investigate the asymptotic behavior of $f$ at infinity,
we shall utilize the Gromov boundary $\del X$.

\subsection{Asymptotic slope}\label{ssc:aslo}%%%%%
%%%%%%%%%%%%%%%%%%%%%%%%%%%%%%%

Define the \emph{descending slope} of $f$ at $x \in X$ by
\[ |\nabla^- f|(x) :=\limsup_{y \to x} \frac{[f(x)-f(y)] \vee 0}{d(x,y)} \in [0,\infty]. \]
For $v \in \del X$ represented by a ray $\xi \colon [0,\infty) \lra X$,
we define the \emph{asymptotic slope}
\[ \del_\infty f(v) :=\lim_{t \to \infty} \frac{f(\xi(t))}{t}
 =\lim_{t \to \infty} \frac{f(\xi(t))-f(\xi(0))}{t} \in (-\infty,\infty]. \]
Note that, as a subsequential limit of geodesics $\gamma_i \colon [0,i] \lra X$ from $\xi(0)$ to $\xi(i)$
along which $f$ is convex, there is a ray $\zeta$ such that $\zeta(0)=\xi(0)$, $\zeta(\infty)=v$,
and $f \circ \zeta$ is convex (recall Lemma~\ref{lm:2geod}).
Then, the function
\begin{equation}\label{eq:aslo-}
t \,\longmapsto\, \frac{f(\zeta(t))-f(\zeta(0))}{t}
\end{equation}
is non-decreasing, thereby the limit as $t \to \infty$ indeed exists.
Moreover, since $\zeta$ is asymptotic to $\xi$ (i.e., $d(\xi(t),\zeta(t))$ is bounded), we have
\[ \biggl| \frac{f(\xi(t))}{t} -\frac{f(\zeta(t))}{t} \biggr|
 \le \frac{Ld(\xi(t),\zeta(t))}{t} \to 0 \]
as $t \to \infty$.
Thus, $\del_\infty f$ is well-defined.
In what follows, we will always choose a ray $\xi$ such that $f \circ \xi$ is convex.

A similar argument also yields the following.

\begin{lemma}\label{lm:bridge}
For any distinct $v_1,v_2 \in \del X$, there exists a geodesic $\xi\colon \R \lra X$
such that $f \circ \xi$ is convex, $\xi_-(\infty) =v_1$ and $\xi_+(\infty) =v_2$,
where we set $\xi_+(t):=\xi(t)$ and $\xi_-(t):=\xi(-t)$ for $t \ge 0$.
\end{lemma}

\begin{proof}
Fix $p \in X$ and let $\zeta_1$ and $\zeta_2$ be rays representing $v_1$ and $v_2$
with $\zeta_1(0)=\zeta_2(0)=p$, respectively.
Given $t>0$, let $\gamma_t \colon [0,1] \lra X$ be a geodesic from $\zeta_1(t)$ to $\zeta_2(t)$
along which $f$ is convex, and $x_t$ be a point on $\gamma_t$ attaining $d(p,\gamma_t)$.
We deduce from \eqref{eq:dxg} that
\[ d(p,x_t) \le \bigl( \zeta_1(t)|\zeta_2(t) \bigr)_p +2\delta. \]
Moreover, since $v_1 \neq v_2$, there is $C>0$ such that $(\zeta_1(t)|\zeta_2(t))_p \le C$ for all $t>0$
(recall Remark~\ref{rm:bdry}).
This implies that $d(p,x_t) \le C+2\delta$ for all $t>0$.
Hence, $\gamma_t$ does not diverge to infinity,
and $\xi$ is obtained as a subsequential limit of a re-parametrization of $\gamma_t$
(e.g., with $\gamma_t(0)=x_t$ and of unit speed).
\end{proof}

We summarize necessary properties of the asymptotic slope $\del_\infty f$ in the next proposition
(cf.\ \cite[Lemma~3.2]{KLM}, \cite[Proposition~2.1]{HS} on $\CAT(0)$-spaces).

\begin{proposition}\label{pr:slope}
Let $(X,d)$ be a proper, geodesic $\delta$-hyperbolic space with $\del X \neq \emptyset$,
and $f\colon X \lra \R$ be an $L$-Lipschitz convex function.
\begin{enumerate}[{\rm (i)}]
\item\label{it:pos}
If $\inf_{v \in \del X} \del_\infty f(v)>0$,
then $f$ is bounded below and its minimum is attained at some point in $X$.

\item\label{it:slope}
We have
\begin{equation}\label{eq:wdual}
\inf_{x \in X} |\nabla^- f|(x) \ge -\inf_{v \in \del X} \del_\infty f(v).
\end{equation}
In particular, if $\inf_{v \in \del X} \del_\infty f(v)<0$, then $\inf_{x \in X} |\nabla^- f|(x)>0$.

\item\label{it:lsc}
$\del_\infty f \colon \del X \lra (-\infty,\infty]$ is lower semi-continuous.

\item\label{it:uniq}
In the case of $\inf_{v \in \del X} \del_\infty f(v)<0$, there exists unique $v_* \in \del X$
such that $\del_\infty f(v_*)<0$.
\end{enumerate}
\end{proposition}

\begin{proof}
\eqref{it:pos}
On the contrary, suppose that there is a sequence $(x_i)_{i \in \N}$ such that $f(x_i) \to -\infty$.
Then, by the Lipschitz continuity of $f$, $d(x_1,x_i) \to \infty$ necessarily holds.
Fix $p \in X$.
Since $f$ is convex, we have $f \le f(p) \vee f(x_i)$ on some geodesic from $p$ to $x_i$.
Thanks to the compactness of $\ol{X}$, as a subsequential limit of those geodesics,
we obtain a ray $\xi$ such that $\xi(0)=p$ and $f(\xi(t)) \le f(p)$ for all $t>0$.
Hence, $\del_\infty f(\xi(\infty)) \le 0$ holds, a contradiction.

The above argument actually shows that each sublevel set of $f$ is bounded.
Then, by the properness of $X$ and the continuity of $f$,
we can find a minimizer of $f$.

\eqref{it:slope}
Fix $x \in X$ and $v \in \del X$ represented by a ray $\xi$ with $\xi(0)=x$.
Then, we deduce from the convexity of $f$ along $\xi$ that
\[ |\nabla^- f|(x) \ge \lim_{t \to 0} \frac{f(x)-f(\xi(t))}{t} \ge \frac{f(x)-f(\xi(t))}{t} \]
for all $t>0$.
Letting $t \to \infty$ shows $|\nabla^- f|(x) \ge -\del_\infty f(v)$, which completes the proof.

\eqref{it:lsc}
Take $p \in X$ and a sequence $(v_i)_{i \in \N}$ in $\del X$ converging to some $v \in \del X$,
and let $\xi_i$ and $\xi$ be rays representing $v_i$ and $v$ with $\xi_i(0)=\xi(0)=p$, respectively.

We first assume that $\del_\infty f(v)<\infty$.
By the definition of $\del_\infty f(v)$, for any $\ve>0$, we have
\[ \frac{f(\xi(t_0)) -f(p)}{t_0} > \del_\infty f(v) -\ve \]
for some $t_0>0$.
Since $f$ is continuous and $\xi_i$ uniformly converges to $\xi$ on each bounded interval,
we find $N \in \N$ such that, for all $i \ge N$,
\[ \frac{f(\xi_i(t_0)) -f(p)}{t_0} > \del_\infty f(v) -2\ve. \]
The monotonicity of the function \eqref{eq:aslo-} then implies $\del_\infty f(v_i) >\del_\infty f(v)-2\ve$,
and hence
\[ \liminf_{i \to \infty} \del_\infty f(v_i) \ge \del_\infty f(v) -2\ve. \]
Letting $\ve \to 0$ completes the proof of the asserted lower semi-continuity:
\[ \liminf_{i \to \infty} \del_\infty f(v_i) \ge \del_\infty f(v). \]

When $\del_\infty f(v)=\infty$, the same argument shows that, for any $R>0$,
\[ \frac{f(\xi(t_0)) -f(p)}{t_0} \ge 2R \]
for some $t_0>0$, and
\[ \del_\infty f(v_i) \ge \frac{f(\xi_i(t_0)) -f(p)}{t_0} \ge R \]
for some $N \in \N$ and all $i \ge N$.
Therefore, we have $\liminf_{i \to \infty} \del_\infty f(v_i) \ge R$ and letting $R \to \infty$
yields $\liminf_{i \to \infty} \del_\infty f(v_i) =\infty$ as desired.

\eqref{it:uniq}
The existence is obvious, thereby it suffices to prove the uniqueness.
If there are distinct $v_1,v_2 \in \del X$ with $\del_\infty f(v_1) \vee \del_\infty f(v_2)<0$,
then, for a geodesic $\xi\colon \R \lra X$ as in Lemma~\ref{lm:bridge},
$f \circ \xi$ is convex and $\lim_{t \to \infty} f(\xi(t))=\lim_{t \to \infty} f(\xi(-t))=-\infty$.
This is impossible, thus we obtain the uniqueness.
\end{proof}

We remark that the uniqueness in \eqref{it:uniq} is a consequence of negative curvature
(cf.\ Example~\ref{ex:Buse}),
and then we readily find that $v_*$ satisfies $\del_\infty f(v_*) =\inf_{\del X} \del_\infty f$
(without applying the lower semi-continuity of $\del_\infty f$).

The inequality \eqref{eq:wdual} is called the \emph{weak duality} in \cite[Lemma~2.2]{HS}.
Then, the \emph{strong duality} means that $\inf_X |\nabla^- f|>0$ if and only if $\inf_{\del X} \del_\infty f<0$.
To show the only if part, in \cite[Lemma~3.4]{KLM} and \cite[Theorems~3.1, 3.7]{HS},
continuous-time gradient flows or a concavity condition on $f$ are used.
In fact, the proof of \cite[Lemma~3.4]{KLM} is applicable to the current setting with the help of \cite{Ly}.

\begin{proposition}[Strong duality]\label{pr:s-dual}
Let $(X,d)$ be a proper, geodesic metric space,
and $f\colon X \lra \R$ be an $L$-Lipschitz convex function.
Then, $\inf_X |\nabla^- f|>0$ if and only if $\del X \neq \emptyset$ and $\inf_{\del X} \del_\infty f<0$.
Moreover, in this case, we have
\[ \inf_X |\nabla^- f| =-\inf_{\del X} \del_\infty f. \]
\end{proposition}

\begin{proof}
We remark that the boundary $\del X$ is defined
and the if part is seen in the same way as Proposition~\ref{pr:slope}\eqref{it:slope},
along with $\inf_X |\nabla^- f| \ge -\inf_{\del X} \del_\infty f$.

In the only if part,
note that $|\nabla^- f|$ is lower semi-continuous (see, e.g., \cite[Corollary~2.4.10]{AGSbook}).
Thus, we can apply \cite[Lemma~6.1]{Ly} (to $-f$) to find a $1$-Lipschitz curve
$\eta\colon [0,\infty) \lra X$ with
\[ \frac{\dd^+}{\dd t}[f \circ \eta](t) =-|\nabla^- f| \bigl( \eta(t) \bigr) \]
for all $t \ge 0$, where $\dd^+/\dd t$ denotes the right derivative.
It then follows that
\[ \frac{f(\eta(t))-f(\eta(0))}{d(\eta(0),\eta(t))}
 \le -\frac{t \cdot \inf_X |\nabla^- f|}{d(\eta(0),\eta(t))}
 \le -\inf_X |\nabla^- f|<0, \]
and $d(\eta(0),\eta(t)) \to \infty$ as $t \to \infty$ since $f$ is Lipschitz.
Hence, for each unit speed geodesic $\gamma_i \colon [0,d(\eta(0),\eta(i))] \lra X$
from $\eta(0)$ to $\eta(i)$ such that $f \circ \gamma_i$ is convex, we have
\[ \frac{f(\gamma_i(t))-f(\eta(0))}{t} \le -\inf_X |\nabla^- f| \]
for all $t \in (0,d(\eta(0),\eta(i))]$.
Therefore, a subsequential limit $\xi$ of $\gamma_i$ satisfies $\del_\infty f(\xi(\infty)) \le -\inf_X |\nabla^- f|$.
This completes the proof.
\end{proof}

\subsection{Hadamard spaces}\label{ssc:CAT0}%%%%%
%%%%%%%%%%%%%%%%%%%%%%%%%%%%%%%

For the sake of comparison, here we briefly discuss the case of nonpositively curved spaces.
We refer to \cite{HS,KLM} for more details.

A geodesic space $(Y,d)$ is called a \emph{$\CAT(0)$-space} if,
for any $x,y,z \in Y$ and any geodesic $\gamma\colon [0,1] \lra Y$ from $y$ to $z$, we have
\[ d^2 \bigl( x,\gamma(s) \bigr) \le (1-s)d^2(x,y) +sd^2(x,z) -(1-s)sd^2(y,z) \]
for all $s \in [0,1]$ (in other words, $d^2(x,\cdot)$ is $2$-convex).
A complete $\CAT(0)$-space is called an \emph{Hadamard space}.
A complete, simply connected Riemannian manifold is a $\CAT(0)$-space
if and only if its sectional curvature is nonpositive everywhere (thus, an \emph{Hadamard manifold}).
Trees and Euclidean buildings are fundamental non-smooth examples of $\CAT(0)$-spaces.

For an Hadamard space $(Y,d)$, one can define $\del Y$
as the equivalence classes of rays in the same manner as Subsection~\ref{ssc:bdry}.
Moreover, $\del Y$ is equipped with a natural metric $\angle_T$ called the \emph{Tits metric},
which can be defined by
\[ 2\sin\biggl( \frac{\angle_T (v_1,v_2)}{2} \biggr) =\lim_{t \to \infty} \frac{d(\xi_1(t),\xi_2(t))}{t} \]
for rays $\xi_1,\xi_2$ associated with $v_1,v_2$.
We call $(\del Y,\angle_T)$ the \emph{Tits boundary}.
For example, when $Y$ is a Euclidean space $\R^n$, then its Tits boundary is isometric to $\Sph^{n-1}$.
The Tits boundary of a hyperbolic space $\Hyp^n$ is discrete
($\angle_T(v_1,v_2)=\pi$ for any distinct $v_1,v_2$).

Given a Lipschitz convex function $f \colon Y \lra \R$ on an Hadamard space,
the asymptotic slope $\del_\infty f \colon \del Y \lra (-\infty,\infty]$ can be defined
as in the previous subsection.
In this case, $\del_\infty f$ is Lipschitz with respect to $\angle_T$
and strictly convex on $\{ v \in \del Y \mid \del_\infty f(v)<0 \}$.
Therefore, when $\inf_{\del Y} \del_\infty f<0$,
we can find a unique minimizer $v_* \in \del Y$ of $\del_\infty f$,
similarly to Proposition~\ref{pr:slope} (see \cite[Lemma~3.2]{KLM}).

\begin{remark}[Recession functions]\label{rm:rec}
In \cite{HS},  $\del_\infty f$ (or its canonical extension to the Euclidean cone $C[\del Y]$)
is called a \emph{recession function} in connection with convex analysis.
\end{remark}

\begin{example}[Busemann functions]\label{ex:Buse}
Given a ray $\xi \colon [0,\infty) \lra Y$ in an Hadamard space,
define the associated \emph{Busemann function} $b_\xi \colon Y \lra \R$ by
\[ b_{\xi}(x) :=\lim_{t \to \infty} \bigl\{ d\bigl( x,\xi(t) \bigr) -t \bigr\}. \]
Inheriting the properties of the distance function, $b_\xi$ is $1$-Lipschitz and convex.
Note also that $b_\xi(\xi(t))=-t$.
A level (resp.\ sublevel) set of $b_\xi$ is called a \emph{horosphere} (resp.\ \emph{horoball}).
If $Y=\R^n$, then we have $b_\xi(x)=-\langle x-\xi(0),\dot{\xi}(0) \rangle$
with the Euclidean inner product $\langle \cdot,\cdot \rangle$,
and horospheres are hyperplanes perpendicular to (the extension of) $\xi$.
When $Y$ is a hyperbolic space of Poincar\'e model,
each horosphere of $b_\xi$ is drawn as a sphere tangent to the boundary at $\xi(\infty)$.
In either case, $v_*$ for $b_\xi$ is given by $\xi(\infty)$.
\end{example}

\section{Discrete-time gradient flows}\label{sc:gf}%%%%%
%%%%%%%%%%%%%%%%%%%%%%%%%%%%%%%

As in Theorem~\ref{th:main},
let $(X,d)$ be a proper, geodesic $\delta$-hyperbolic space with $\del X \neq \emptyset$,
and $f\colon X \lra \R$ be an $L$-Lipschitz convex function.

\subsection{Proximal point algorithm}\label{ssc:PPA}%%%%%
%%%%%%%%%%%%%%%%%%%%%%%%%%%%%%%

For $\tau>0$ and $x \in X$, the \emph{proximal}  (or \emph{resolvent}) \emph{operator} is defined by
\[ \sJ^f_{\tau}(x) :=\argmin_{y \in X} \bigg\{ f(y) +\frac{d^2(x,y)}{2\tau} \bigg\}. \]
Roughly speaking, an element in $\sJ^f_{\tau}(x)$ can be regarded as
an approximation of a point on the gradient curve of $f$ at time $\tau$ from $x$.
Note that $\sJ^f_{\tau}(x) \neq \emptyset$ by the properness of $(X,d)$
(see also the beginning of \cite[\S 3.1]{Ohyp}).

For any $x \in X$ and $x_\tau \in \sJ^f_{\tau}(x)$,
we infer from the $L$-Lipschitz continuity of $f$ that
\[ f(x_\tau) +\frac{d^2(x,x_\tau)}{2\tau} \le f(x) \le f(x_\tau) +Ld(x,x_\tau). \]
This implies
\begin{equation}\label{eq:xxt<}
d(x,x_\tau) \le 2\tau L.
\end{equation}
Using the convexity of $f$, we can also provide a lower bound of $d(x,x_\tau)$.

\begin{lemma}
Suppose that $\alpha :=-\inf_{\del X} \del_\infty f>0$.
Then, for any $x \in X$ and $x_\tau \in \sJ^f_{\tau}(x)$, we have
\begin{equation}\label{eq:xxt>}
d(x,x_\tau) \ge \bigl( L-\sqrt{L^2 -\alpha^2} \bigr) \tau \ge \frac{\alpha^2 \tau}{2L}.
\end{equation}
Moreover,
\begin{equation}\label{eq:fxtau}
f(x_\tau) \le f(x) -\frac{\alpha^4 \tau}{8L^2}.
\end{equation}
\end{lemma}

We remark that, by \eqref{eq:wdual},
\[ L \ge \inf_X |\nabla^- f| \ge -\inf_{\del X} \del_\infty f =\alpha. \]

\begin{proof}
Let $v_* \in \del X$ be the unique element satisfying $\del_\infty f(v_*)=-\alpha$,
and $\xi$ be a ray with $\xi(0)=x$ and $\xi(\infty)=v_*$.
By the monotonicity of the function \eqref{eq:aslo-}, we have
\[ \frac{f(\xi(t))-f(x)}{t} \le -\alpha \]
for all $t>0$.
Since $d(x,\xi(t))=t$, the above inequality is rewritten as
\[ f\bigl( \xi(t) \bigr) +\frac{d^2(x,\xi(t))}{2\tau} \le f(x) -\alpha t +\frac{t^2}{2\tau}, \]
and the right hand side takes its minimum at $t=\alpha\tau$, yielding
\begin{equation}\label{eq:atau}
f\bigl( \xi(\alpha\tau) \bigr) +\frac{d^2(x,\xi(\alpha\tau))}{2\tau} \le f(x) -\frac{\alpha^2 \tau}{2}.
\end{equation}

Now, for any $y \in X$ with $d(x,y)<(L-\sqrt{L^2 -\alpha^2})\tau$,
we deduce from the $L$-Lipschitz continuity of $f$ that
\begin{align*}
f(y) +\frac{d^2(x,y)}{2\tau}
&\ge f(x) -Ld(x,y) +\frac{d^2(x,y)}{2\tau} \\
&= f(x) +\frac{1}{2\tau} \bigl\{ \bigl( L\tau -d(x,y) \bigr)^2 -L^2 \tau^2 \bigr\} \\
&> f(x) +\frac{1}{2\tau} \bigl\{ (L^2 -\alpha^2) \tau^2 -L^2 \tau^2 \bigr\} \\
&= f(x) -\frac{\alpha^2 \tau}{2}.
\end{align*}
Comparing this with \eqref{eq:atau} shows that $y \not\in \sJ^f_{\tau}(x)$,
thereby the former inequality $d(x,x_\tau) \ge (L-\sqrt{L^2 -\alpha^2})\tau$ in \eqref{eq:xxt>} necessarily holds.
The latter inequality in \eqref{eq:xxt>} is immediate.

The second assertion \eqref{eq:fxtau} follows from the choice of $x_\tau$ together with \eqref{eq:xxt>} as
\[ f(x) \ge f(x_\tau) +\frac{d^2(x,x_\tau)}{2\tau} \ge f(x_\tau) +\frac{\alpha^4 \tau}{8L^2}. \]
\end{proof}

We remark that the $\delta$-hyperbolicity was not used in the lemma above.
It will come into play via the following estimate from \cite{Ohyp} (we state only the case of $K=0$).

\begin{theorem}[\cite{Ohyp}]\label{th:to-p}
Let $(X,d)$ be a proper, geodesic $\delta$-hyperbolic space
and $f\colon X \lra \R$ be an $L$-Lipschitz convex function.
Then, for any $x \in X$, $y \in \sJ^f_{\tau}(x)$, and $p \in X$ with $f(p) \le f(y)$, we have
\begin{equation}\label{eq:to-p}
d(p,y) \le d(p,x) -d(x,y) +4\sqrt{2\tau L\delta}.
\end{equation}
\end{theorem}

We remark that, in \cite[Theorem~1.1]{Ohyp}, $p$ was chosen as a minimizer of $f$;
however, by having a look on its proof, it is sufficient to assume $f(p) \le f(y)$.

\begin{remark}[No semi-contraction]\label{rm:nocont}
We also obtained a kind of contraction property in \cite{Ohyp},
whereas it does not imply that the proximal operator is a semi-contraction.
For $y_i \in \sJ^f_{\tau}(x_i)$ ($i=1,2$) and any minimizer $p \in X$ of $f$,
\cite[Theorem~1.3(ii)]{Ohyp} (with $K=0$) asserts that
\[ d(y_1,y_2) \le d(x_1,x_2) -(p|x_2)_{x_1} +C(L,D,\tau,\delta), \]
where $D:=d(p,x_1) \vee d(p,x_2)$, provided $d(p,y_1) \le d(p,y_2) \wedge (x_1|x_2)_p$.
If $x_2 \to x_1$, then $d(x_1,x_2) -(p|x_2)_{x_1} \to 0$,
but the additional term $C(L,D,\tau,\delta)$ caused by the $\delta$-hyperbolicity remains.
\end{remark}

\subsection{Proof of Theorem~\ref{th:main}}\label{ssc:proof}%%%%%
%%%%%%%%%%%%%%%%%%%%%%%%%%%%%%%

We have shown \eqref{it:v*} in Proposition~\ref{pr:slope},
here we prove the remaining assertions \eqref{it:to-v*}, \eqref{it:quan}.

\begin{proof}[Proof of Theorem~$\ref{th:main}$]
\eqref{it:to-v*}
Fix an arbitrary initial point $x_0 \in X$ and recursively choose $x_k \in \sJ^f_{\tau}(x_{k-1})$ for $k \in \N$.
It follows from \eqref{eq:xxt<} and \eqref{eq:fxtau} that
\[ d(x_0,x_k) \le 2k\tau L, \qquad
 f(x_k) \le f(x_0) -k\frac{\alpha^4 \tau}{8L^2}, \]
and hence
\begin{equation}\label{eq:f/d}
\frac{f(x_k)-f(x_0)}{d(x_0,x_k)} \le -k\frac{\alpha^4 \tau}{8L^2} \frac{1}{2k\tau L}
 =-\frac{\alpha^4}{16L^3}.
\end{equation}
Note also that
\begin{equation}\label{eq:x0xk}
d(x_0,x_k) \ge \frac{f(x_0)-f(x_k)}{L} \ge k\frac{\alpha^4 \tau}{8L^3}.
\end{equation}
Let $\gamma_k \colon [0,d(x_0,x_k)] \lra X$ be a unit speed geodesic
from $x_0$ to $x_k$ on which $f$ is convex.
Then, we deduce from \eqref{eq:f/d} that
\[ \frac{f(\gamma_k(t))-f(x_0)}{t} \le \frac{f(x_k)-f(x_0)}{d(x_0,x_k)}
 \le -\frac{\alpha^4}{16L^3} \]
for all $t \in (0,d(x_0,x_k)]$, and \eqref{eq:x0xk} ensures that
$d(x_0,x_k) \to \infty$ as $k \to \infty$.
Therefore, a subsequence of $\gamma_k$ converges to a ray $\xi\colon [0,\infty) \lra X$
such that $\del_\infty f\bigl( \xi(\infty) \bigr) \le -\alpha^4/(16L^3)<0$,
and then $\xi(\infty)=v_*$ by Proposition~\ref{pr:slope}\eqref{it:uniq}.
Since the limit is unique, the original sequence $(x_k)_{k \in \N}$ also converges to $v_*$.

\eqref{it:quan}
Let $\xi\colon [0,\infty) \lra X$ be a ray with $\xi(0)=x_0$ and $\xi(\infty)=v_*$.
Observe that \eqref{eq:x0xk} implies
\begin{equation}\label{eq:d>}
d\bigl( \xi(t),x_k \bigr) =d\bigl( x_0,\xi(t) \bigr) +d(x_0,x_k) -2\bigl( \xi(t)|x_k \bigr)_{x_0}
 \ge t +k\frac{\alpha^4 \tau}{8L^3} -2\bigl( \xi(t)|x_k \bigr)_{x_0}
\end{equation}
for all $k \in \N$ and $t \ge 0$.
On the other hand, it follows from \eqref{eq:to-p} and \eqref{eq:xxt>} that,
for $t>0$ satisfying $f(\xi(t)) \le f(x_k)$ (note that $f(\xi(t)) \to -\infty$ as $t \to \infty$),
\[ d\bigl( \xi(t),x_k \bigr) -d\bigl( \xi(t),x_{k-1} \bigr)
 \le -d(x_{k-1},x_k) +4\sqrt{2\tau L\delta}
 \le -\frac{\alpha^2 \tau}{2L} +4\sqrt{2\tau L\delta}. \]
Therefore, for $t>0$ with $f(\xi(t)) \le f(x_k)$,
\[ d\bigl( \xi(t),x_k \bigr) \le d\bigl( \xi(t),x_0 \bigr)
 -k\biggl( \frac{\alpha^2 \tau}{2L} -4\sqrt{2\tau L\delta} \biggr). \]
Combining this with \eqref{eq:d>} yields
\[ k\biggl( \frac{\alpha^2 \tau}{2L} -4\sqrt{2\tau L\delta} \biggr)
 \le t-d\bigl( \xi(t),x_k \bigr)
 \le -k\frac{\alpha^4 \tau}{8L^3} +2\bigl( \xi(t)|x_k \bigr)_{x_0}, \]
thereby
\[ k\sqrt{\tau}
 \biggl\{ \biggl( \frac{\alpha^2}{2L} +\frac{\alpha^4}{8L^3} \biggr) \sqrt{\tau} -4\sqrt{2L\delta} \biggr\}
 \le 2\bigl( \xi(t)|x_k \bigr)_{x_0}. \]
This completes the proof.
\end{proof}

\subsection{Further problems}\label{ssc:prob}%%%%%
%%%%%%%%%%%%%%%%%%%%%%%%%%%%%%%

We close the article with some further problems, including those discussed at the end of \cite{Ohyp}.

\begin{enumerate}[(A)]
\item
As we explained in Example~\ref{ex:Ghyp}\eqref{it:disc},
the Gromov hyperbolicity makes sense for discrete spaces as well.
Therefore, it is interesting to explore some generalizations of the results in this article and \cite{Ohyp}
to discrete (non-geodesic) Gromov hyperbolic spaces.
Then, it is a challenging problem to formulate and analyze
convex functions on discrete Gromov hyperbolic spaces
(possibly for some special classes such as hyperbolic groups).
We refer to \cite{Mu} for the theory of convex functions on $\mathbb{Z}^N$
(called \emph{discrete convex analysis}),
and to \cite{Hi1,Ko} for some generalizations to graphs and trees, respectively.

\item
Even in geodesic Gromov hyperbolic spaces,
it is worthwhile considering a certain ``large-scale convexity'' of functions,
preserved by \emph{quasi-isometries},
since the Gromov hyperbolicity is preserved by quasi-isometries between geodesic spaces
(see, e.g., \cite[Theorem~III.H.1.9]{BH}, \cite[Theorem~3.18]{Va}).

\item
Since \cite{HS,KLM,KM} are concerned with nonpositively curved spaces,
it is natural to expect that our results can be extended to some class of metric spaces
including both $\CAT(0)$-spaces and Gromov hyperbolic spaces,
probably defined through an appropriate relaxation (perturbation) of the $\CAT(0)$-condition.
\end{enumerate}
%\medskip

\textit{Acknowledgements.}
I would like to thank Hiroshi Hirai for stimulating discussions.
I am also grateful to an anonymous referee for suggesting Lemma~\ref{lm:bridge} and Proposition~\ref{pr:s-dual}
and for his/her valuable comment on the proof of Theorem~\ref{th:main}.
This work was supported in part by the JSPS Grant-in-Aid for Scientific Research (KAKENHI)
22H04942, 24K00523, 24K21511.

{\small%%%

}

\end{document}